\documentclass[11pt]{article}
\usepackage{amsmath}
\usepackage{amsthm}
\usepackage{amsfonts} 

\usepackage{mathabx}
\usepackage{amssymb}
\usepackage[pdftex]{graphicx} 
\usepackage[english]{babel} 
\usepackage[colorlinks = true,
            linkcolor = blue,
            urlcolor  = blue,
            citecolor = blue,
            anchorcolor = blue]{hyperref} 
\usepackage{calc} 
\usepackage{dsfont}
\usepackage{enumitem}

\usepackage{authblk}
\usepackage{physics}
\usepackage{tikz}
\usetikzlibrary{braids}
\usetikzlibrary{knots}

\allowdisplaybreaks
\usepackage[a4paper, lmargin=0.1666\paperwidth, rmargin=0.1666\paperwidth, tmargin=0.1111\paperheight, bmargin=0.1111\paperheight]{geometry} 
\usepackage[all]{nowidow}
\usepackage{lipsum} 
\hypersetup{ 	
	pdfsubject = {},
	pdftitle = {},
	pdfauthor = {}
}

\theoremstyle{plain}
\newtheorem{theorem}{Theorem}[section]

\newtheorem{corollary}{Corollary}[section]

\theoremstyle{definition}
\newtheorem{definition}[theorem]{Definition}

\theoremstyle{remark}

\newtheoremstyle{named}{}{}{\itshape}{}{\bfseries}{.}{.5em}{\thmnote{#3}#1}
\theoremstyle{named}

\newcommand{\br}{\mathrm{br}}

\title{Petal diagram from simple braids}
\author{Zipei Nie\thanks{Part of the work was done when the second author visited Institut des Hautes Études Scientifiques.}}
\affil{
  LMCRC, Huawei\\
\texttt{niezipei@huawei.com}
}
	
\date{\today}

\begin{document} 
\maketitle
\begin{abstract}
   We construct petal diagrams from simple braids. This approach allows us to confirm a conjecture proposed by Kim, No and Yoo, which states that the petal number of the nontrivial torus knot $T_{r,s}$ ($r<s$) is at most $2s-2\lfloor\frac{s}{r}\rfloor+1$. As a consequence, we deduce that the petal number of a nontrivial torus knot $T_{r,s}$ is equal to $2s-1$ if and only if $r<s<2r$.
\end{abstract}
\section{Introduction}
A petal diagram, as introduced in \cite{adams2015knot}, is a special knot diagram. In contrast to the standard knot diagram, which has multiple crossings, each involving one overstrand and one understrand, the petal diagram has a single multi-crossing \cite{adams2013triple}, at which multiple strands intersect. Additionally, a petal diagram forms a topological rose. According to \cite[Corollary 2.3]{adams2015knot}, every knot admits a petal diagram. The petal number of a knot is defined as the minimum number of petals of the topological rose formed by a petal diagram.

The petal numbers of nontrivial torus knots are of particular interest. By \cite[Theorem 4.15]{adams2015knot}, the petal number of the nontrivial torus knot $T_{r,s}$ is at most $2s-1$ when $s\equiv 1\pmod{r}$ and at most $2s+3$ when $s\equiv -1\pmod{r}$. The proof involves the explicit construction of petal diagrams, and the result successfully determines the petal number of $T_{r,r+1}$. Moreover, Lee and Jin \cite{lee2021petal} constructed a petal diagram for $T_{r,r+2}$ and thereby determined its petal number. These results were further refined \cite{kim2022petal} by Kim, No and Yoo. They proved that the petal number of $T_{r,s}$ is at most $2s-2\lfloor \frac{s}{r} \rfloor+1$ when $s\equiv \pm 1\pmod{r}$, and used the superbridge index to determine the petal number of $T_{r,s}$ when $r<s$ and $r\equiv 1\pmod{s-r}$. Furthermore, they proposed a conjecture that their upper bound of the petal number applies to any nontrivial torus knot $T_{r,s}$ with $r<s$. This paper aims to prove this conjecture.

A (positive) simple braid, also known as a permutation braid or a non-repeating braid, is a positive braid in which each pair of strands crosses at most once. There exists a bijection between the symmetric group $S_n$ and the set of simple braids on $n$ strands, as shown in \cite[Lemma 9.10]{epstein1992word}. We denote the simple braid corresponding to the permutation $\pi \in S_n$ as $\br(\pi)$. Let $\mathrm{Inv}(\pi)$ denote the inversion set of $\pi$. When $\mathrm{Inv}(\pi_1\pi_2)$ contains $\mathrm{Inv}(\pi_2)$, we have $\br(\pi_1)\br(\pi_2)=\br(\pi_1\pi_2)$. 

Our main contribution is the construction of petal diagrams from simple braids. We prove the following theorem in Section~\ref{sec:2}.

\begin{theorem}\label{main-1}
    For any permutations $\pi_1,\pi_2\in S_n$, the petal number of the closure of the braid \[\delta_n\br(\pi_1)^{-1}\br(\pi_1^{-1})^{-1} \br(\pi_2^{-1})\br(\pi_2)\] is at most $2n+1$, where $\delta_n$ is the simple braid \[\br\left(\left(\begin{array}{ccccc}
1 & 2 & \cdots & n-1 & n  \\
2 & 3 & \cdots & n & 1 
\end{array}\right)\right).\]
\end{theorem}

In Section~\ref{sec:3}, we represent a nontrivial torus knot as the closure of a braid, and use Theorem~\ref{main-1} to establish the following upper bound on the petal number.

\begin{theorem}\label{main-2}
    Let $r$ and $s$ be two relatively prime numbers with $1<r<s$. The petal number of the torus knot $T_{r,s}$ is at most $2s-2\left\lfloor\frac{s}{r}\right\rfloor+1$.
\end{theorem}

Combined with the lower bound \cite[Theorem 2]{kim2022petal} derived from the superbridge index, we determine the petal number of torus knots $T_{r,s}$ with $r < s < 2r$, thus improving upon \cite[Theorem 1.2]{lee2021petal} and \cite[Theorem 1]{kim2022petal}.

\begin{corollary}
    The petal number of a nontrivial torus knot $T_{r,s}$ is equal to $2s-1$ if and only if $r< s <2r$.
\end{corollary}
\begin{proof}   
    We consider three scenarios:
    \begin{enumerate}[label=(\alph*)]
        \item Suppose that $s < r$. According to \cite[Theorem B]{kuiper1987new}, the superbridge index of $T_{r,s}$ is $\min(r,2s)$. By \cite[Theorem 2]{kim2022petal}, the petal number is at least $2\min(r,2s)-1$, which is greater than $2s-1$.

\item Suppose that $r < s < 2r$. According to \cite[Theorem B]{kuiper1987new}, the superbridge index of $T_{r,s}$ is $s$. Then by Theorem~\ref{main-2} and \cite[Theorem 2]{kim2022petal}, the petal number is equal to $2s-1$.

\item Suppose that $s > 2r$. By Theorem~\ref{main-2}, the petal number is at most $2s-3$.
    \end{enumerate}
\end{proof}

\section{Construction of petal diagrams}\label{sec:2}
We consider the braid analogue of petal diagrams.

\begin{definition}
    A braid diagram is called a star diagram if and only if it is connected and has a single multi-crossing.
\end{definition}

A star diagram always forms a topological star. Its isotopy class is determined by the relative $z$-coordinates of the strands. We give a characterization of the braids admitting star diagrams.

\begin{theorem}\label{star_diagram}
    A braid on $n$ $(n\ge 2)$ strands admits a star diagram if and only if it can be represented as \[\Delta_n^{-1}\br(\pi^{-1})\br(\pi)\] for some permutation $\pi\in S_n$, where $\Delta_n$ is the positive half-twist on $n$ strands.
\end{theorem}
\begin{proof}
   We depict the braid diagram from left to right horizontally; see Figure~\ref{fig:1}.

\begin{figure}[!htbp] 
\centering
\begin{tikzpicture}   
\draw (1,-4) node[left] {$s_4$};
\draw (1,-3) node[left] {$s_3$};
\draw (1,-2) node[left] {$s_2$};
\draw (1,-1) node[left] {$s_1$};

\begin{knot}[
consider self intersections,clip width=10,
flip crossing=2
]
\strand[thick,black]
(1,-1) to (5,-4);
\strand[thick,black]
(1,-2) to (5,-3);
\strand[thick,black]
(1,-3) to (5,-2);
\strand[thick,black]
(1,-4) to (5,-1);
\end{knot}
\draw (7,-4) node[left] {$s_4$};
\draw (7,-3) node[left] {$s_3$};
\draw (7,-2) node[left] {$s_2$};
\draw (7,-1) node[left] {$s_1$};
\draw [-stealth,thick](5.2,-2.5) -- (6.3,-2.5);
\begin{knot}[
consider self intersections,clip width=10,
flip crossing=3,
flip crossing=4,
flip crossing=5
]
\strand[thick,black]
(7,-1) to (9,-3) to (11,-4);
\strand[thick,black]
(7,-2)  to (9,-1) to (11,-3);
\strand[thick,black]
(7,-3) to (9, -4) to (11,-2);
\strand[thick,black]
(7,-4) to (9,-2) to (11,-1);
\end{knot}
\end{tikzpicture}
\caption{Star diagram with $\pi= \left(\begin{array}{cccc}
1 & 2 & 3 & 4  \\
3 & 1 & 4 & 2  \\
\end{array}\right)$.}\label{fig:1} 
\end{figure}
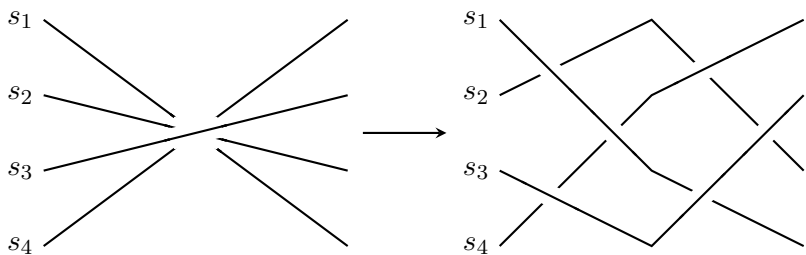

   Let $s_1, s_2, \ldots, s_n$ represent the $n$ strands in a star diagram, arranged naturally such that $s_1$ runs from the top-left to the bottom-right. Assume that the midpoint of each $s_i$ coincides with the location of the multi-crossing. Consider a permutation $\pi \in S_n$ where the $z$-coordinate of the midpoint of $s_i$ is greater than the $z$-coordinate of $s_j$ if and only if $\pi(i) > \pi(j)$. Then, the braid represented by the star diagram is uniquely determined by $\pi$.

    We vertically shift the midpoints of the strands such that the midpoint of $s_i$ is lower than that of $s_j$ if and only if $\pi(i) > \pi(j)$. The resulting braid diagram has positive crossings on the left and negative crossings on the right. Moreover, it represents the product\footnote{We adopt Thurston's convention that the braid product $ab$ starts with $b$.} of the inverse of a simple braid $\br(\pi_0)$ and the simple braid $\br(\pi)$. Since $\br(\pi)$ permutes the strands with respect to $\pi$, and $\br(\pi_0)^{-1}$ permutes the strands with respect to $\pi_0^{-1}$, we have
    \[\pi_0^{-1}\pi=\left(\begin{array}{cccc}
1 & 2 & \cdots & n  \\
n & n-1 & \cdots & 1  \\
\end{array}\right).\] 
Because $\mathrm{Inv}(\pi^{-1}\pi_0)$ contains $\mathrm{Inv}(\pi_0)$, we find that $\br(\pi^{-1})\br(\pi_0)=\Delta_n$. Therefore, the braid diagram represents the braid \[\br(\pi_0)^{-1}\br(\pi)=\Delta_n^{-1}\br(\pi^{-1})\br(\pi).\]
\end{proof}

Now we prove Theorem~\ref{main-1} by constructing a petal diagram of the braid closure.

\begin{proof}[Proof of Theorem~\ref{main-1}]
    According to Theorem~\ref{star_diagram}, the braids \[\Delta_n^{-1}\br(\pi_1^{-1})\br(\pi_1)\] and \[\Delta_n^{-1}\br(\pi_2^{-1})\br(\pi_2)\] admit star diagrams. Since the inverse of a braid admitting a star diagram also admits a star diagram, the braid \[\br(\pi_1)^{-1}\br(\pi_1^{-1})^{-1}\Delta_n\] also admits a star diagram. Thus, we can depict a braid diagram for \[\delta_n\br(\pi_1)^{-1}\br(\pi_1^{-1})^{-1} \br(\pi_2^{-1})\br(\pi_2) \] by concatenating two star diagrams and a standard diagram of the simple braid $\delta_n$ as shown in Figure~\ref{fig:2}.

\begin{figure}[!htbp] 
\centering
\begin{tikzpicture}

\begin{knot}[
consider self intersections,clip width=10
]
\strand[thick,black]
 (9,-1)to (11,-2);
\strand[thick,black]
(9,-2)to (11,-3);
\strand[thick,black]
(9,-3)to (11,-4);
\strand[thick,black]
(9,-4)to (11,-1);
\end{knot}
\draw[thick,black](1,-1) to (5,-4)to (9,-1);
\draw[thick,black](1,-2) to (5,-3)to (9,-2);
\draw[thick,black](1,-3) to (5,-2)to (9,-3);
\draw[thick,black](1,-4) to (5,-1)to (9,-4);
\end{tikzpicture}
\caption{Concatenation of two star diagrams and $\delta_n$.}\label{fig:2} 
\end{figure}
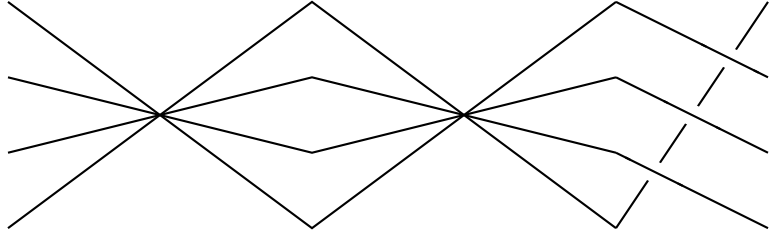

We fold the first star diagram outward to align it with the second one so that two multi-crossings coincide. Then, we close the braid. As illustrated in Figure~\ref{fig:3}, the $z$-coordinate of the blue strand is greater than that of every red strand and less than that of every black strand. Therefore, we can move the blue strand to the middle to obtain a petal diagram with $2n+1$ petals.

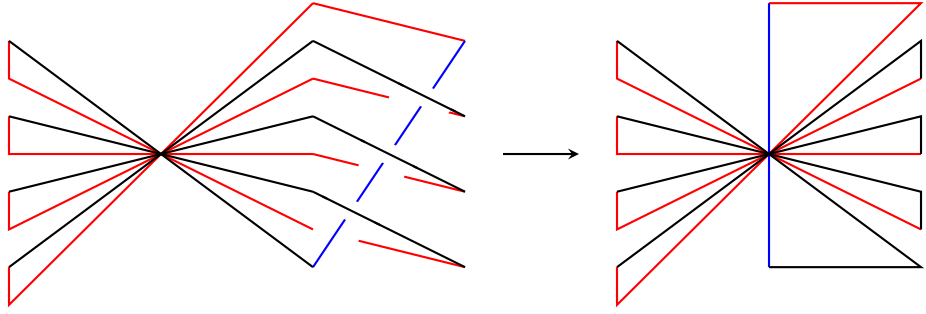
\begin{figure}[!htbp] 
\centering
\begin{tikzpicture}   

\draw[thick,red]
(5,-0.5) to (7,-1);
\draw[thick,red]
(5,-1.5) to (6,-1.75);
\draw[thick,red]
(6.6,-1.9) to (7,-2);
\draw[thick,red]
(5,-2.5) to (5.6,-2.65);
\draw[thick,red]
(6.2,-2.8) to (7,-3);
\draw[thick,red]
(5.6,-3.65) to (7,-4);

\begin{knot}[
consider self intersections,clip width=10,
flip crossing=1,
flip crossing=2,
flip crossing=3
]
\strand[thick,blue]
(5,-4) to (7,-1);
\strand[thick,black]
(5,-3) to (7,-4);
\strand[thick,black]
(5,-2) to (7,-3);
\strand[thick,black]
(5,-1) to (7,-2);
\end{knot}

\draw[thick,red]
(1,-1) to (1,-1.5) to (5,-3.5);
\draw[thick,red]
(1,-2) to (1,-2.5) to (5,-2.5);
\draw[thick,red]
(1,-3) to (1,-3.5) to (5,-1.5);
\draw[thick,red]
(1,-4) to (1,-4.5) to (5,-0.5);

\draw[thick,black](1,-1) to (5,-4);
\draw[thick,black](1,-2) to (5,-3);
\draw[thick,black](1,-3) to (5,-2);
\draw[thick,black](1,-4) to (5,-1);

\draw [-stealth,thick](7.5,-2.5) -- (8.5,-2.5);
\draw[thick,red]
(9,-1) to (9,-1.5) to (13,-3.5);
\draw[thick,red]
(9,-2) to (9,-2.5) to (13,-2.5);
\draw[thick,red]
(9,-3) to (9,-3.5) to (13,-1.5);
\draw[thick,red]
(9,-4) to (9,-4.5) to (13,-0.5)to (11,-0.5);
\draw[thick,blue]
(11,-4) to (11,-0.5);
\draw[thick,black](9,-1) to (13,-4)to (11,-4) ;
\draw[thick,black](9,-2) to (13,-3) to(13,-3.5);
\draw[thick,black](9,-3) to (13,-2) to(13,-2.5);
\draw[thick,black](9,-4) to (13,-1)to (13,-1.5);

\end{tikzpicture}
\caption{A petal diagram of the braid closure.}\label{fig:3} 
\end{figure}
\end{proof}

\section{Proof of Theorem~\ref{main-2}}\label{sec:3}
As shown in Figure~\ref{fig:4}, the standard braid diagram for the torus knot $T_{r,s}$ on $s$ strands can be represented as a product of two simple braids \[\br(\pi_1)\br(\pi_2),\]
where \[
\pi_1(i):=\begin{cases}
    r+1-i&\mbox{, if } 1\le i\le r,\\
    i &\mbox{, if } r+1\le i\le s,
\end{cases}
\]
and \[
\pi_2(i):=\begin{cases}
    r+i& \mbox{, if }1\le i\le s-r,\\
    s+1-i &\mbox{, if } s-r+1\le i\le s.
\end{cases}
\]
Then we have $\pi_1=\pi_1^{-1}.$
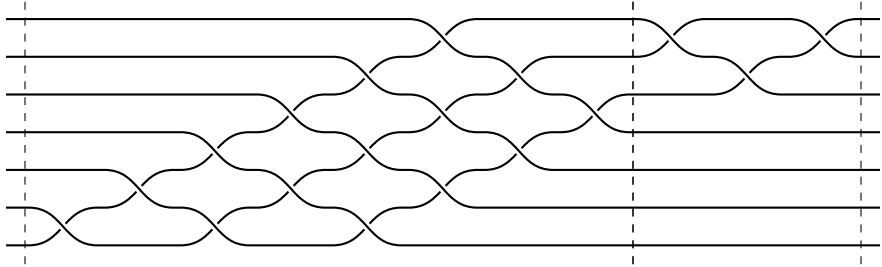
\begin{figure}[!htbp] 
\begin{center}
\begin{tikzpicture}[rotate=90,transform shape] 
\pic[braid/every floor/.style={draw=black,dashed},
braid/add floor={1,0,6,8,a},
braid/add floor={1,8,6,3,b},
braid/.cd,
every strand/.style={thick},width=0.5cm,
] {braid={ s_1^{-1} s_2^{-1} s_1^{-1}-s_3^{-1} s_4^{-1}-s_2^{-1} s_5^{-1}-s_3^{-1}-s_1^{-1} s_6^{-1}-s_4^{-1}-s_2^{-1} s_5^{-1}-s_3^{-1} s_4^{-1} s_6^{-1} s_5^{-1} s_6^{-1} }};
\end{tikzpicture}
\end{center}
\caption{Torus knot $T_{3,7}$ as the closure of the product of two simple braids.}\label{fig:4} 
\end{figure}
 
Because $r$ and $s$ are relatively prime, there exists a permutation $\pi_3\in S_s$ such that $\pi_3(i):=ri\pmod{s}$ for each $i=1, 2, \ldots, s$. Define the cyclic permutation $\pi_4$ as $\left(\begin{array}{ccccc}
1 & 2 & \cdots & s-1 & s  \\
2 & 3 & \cdots & s & 1  \\
\end{array}\right)$ in $S_s.$ We can observe that $\pi_3 \pi_4 = \pi_4^{r}\pi_3$ and $\pi_2= \pi_1\pi_4^r.$ By $\pi_3(s)=s$, the inversion set $\mathrm{Inv}(\pi_4 \pi_3^{-1})$ contains $\mathrm{Inv}(\pi_3^{-1})$. Therefore we have \[\br(\pi_4 \pi_3^{-1})=\br(\pi_4) \br(\pi_3^{-1})=\delta_s \br(\pi_3^{-1}).\]
By Markov's theorem, applying the conjugate by $\br(\pi_4 \pi_3^{-1})^{-1}$ does not alter the braid closure. Therefore, the closure of the braid
\[\delta_s\br(\pi_3^{-1})\br(\pi_1)\br(\pi_2)\br(\pi_4 \pi_3^{-1})^{-1}\] is $T_{r,s}$.

Define the permutation $\pi_5\in S_s$ by 
\[\pi_5^{-1}(i)=\begin{cases} 
 i &\mbox{, if } 1\le i\le r,\\
|\{\pi_3^{-1}(i)\ge \pi_3^{-1}(j): r+1\le j\le s\}|+r &\mbox{, if } r+1\le i\le s.
\end{cases}\] Then $\pi_1\pi_5=\pi_5\pi_1$. For each $r+1\le i,j\le s$, we have $\pi_3^{-1}(i)\le \pi_3^{-1}(j)$ if and only if $\pi_5^{-1}(i)\le \pi_5^{-1}(j)$. So $\mathrm{Inv}(\pi_5^{-1} \pi_2)$ contains $\mathrm{Inv}(\pi_3^{-1} \pi_4^{r})$, and we have
\[\br(\pi_5^{-1} \pi_2)= \br(\pi_5^{-1} \pi_2\pi_4^{-r}\pi_3)\br(\pi_3^{-1} \pi_4^{r}).\]
Because $\mathrm{Inv}(\pi_3^{-1})$ contains $\mathrm{Inv}(\pi_5^{-1})$, we have 
\[\br(\pi_3^{-1})=\br(\pi_3^{-1}\pi_5)\br(\pi_5^{-1}).\]
Because $\mathrm{Inv}(\pi_5^{-1}\pi_2)$ contains $\mathrm{Inv}(\pi_2)$, we have 
\[\br(\pi_5^{-1} \pi_2)=\br(\pi_5^{-1})\br(\pi_2).\]
Because $\mathrm{Inv}(\pi_5^{-1}\pi_1)$ contains $\mathrm{Inv}(\pi_1)$, we have 
\[\br(\pi_5^{-1} )=\br(\pi_5^{-1})\br(\pi_1).\]
Because $\mathrm{Inv}(\pi_1\pi_5^{-1})$ contains $\mathrm{Inv}(\pi_5^{-1})$, we have 
\[\br(\pi_1\pi_5^{-1})=\br(\pi_1)\br(\pi_5^{-1}).\]

Define the permutation $\pi_6\in S_s$ by 
\[\pi_6^{-1}(i)=\begin{cases} 
|\{\pi_3^{-1}(i)\ge \pi_3^{-1}(j): 1\le j\le r\}|&\mbox{, if }  1\le i\le r,\\
    i &\mbox{, if }, r+1\le i\le s.
\end{cases}\] 
 For each $1\le i,j\le r$, we have $\pi_3^{-1}(i)\le \pi_3^{-1}(j)$ if and only if $\pi_6^{-1}(i)\le \pi_6^{-1}(j)$. Because every pair $(i,j)$ with $1\le i<j\le r$ is in $\mathrm{Inv}(\pi_3^{-1}\pi_5\pi_1\pi_6)$, the inversion set $\mathrm{Inv}(\pi_3^{-1}\pi_5\pi_1\pi_6)$ contains $\mathrm{Inv}(\pi_1\pi_6)$, which implies
 \[\br(\pi_3^{-1}\pi_5\pi_1\pi_6) = \br(\pi_3^{-1}\pi_5)\br(\pi_1\pi_6).\]
 Similarly, the inversion set $\mathrm{Inv}(\pi_6^{-1}\pi_1^{-1}\pi_5^{-1}\pi_3)$ contains  $\mathrm{Inv}(\pi_1^{-1}\pi_5^{-1}\pi_3)$, which implies
 \[\br(\pi_6^{-1}\pi_1^{-1}\pi_5^{-1}\pi_3) = \br(\pi_6^{-1})\br(\pi_1^{-1}\pi_5^{-1}\pi_3).\]
Because $\mathrm{Inv}(\pi_1)$ contains $\mathrm{Inv}(\pi_6^{-1})$, we have \[\br(\pi_1)= \br(\pi_1\pi_6)\br(\pi_6^{-1}).\]

In summary, we have

\begin{align*}
    &\delta_s\br(\pi_3^{-1})\br(\pi_1)\br(\pi_2)\br(\pi_4 \pi_3^{-1})^{-1}\\
    =&\delta_s\br(\pi_3^{-1}\pi_5)\br(\pi_5^{-1})\br(\pi_1)\br(\pi_2)\br(\pi_4 \pi_3^{-1})^{-1}\\
    =&\delta_s\br(\pi_3^{-1}\pi_5)\br(\pi_5^{-1}\pi_1)\br(\pi_2)\br(\pi_4 \pi_3^{-1})^{-1}\\
    =&\delta_s\br(\pi_3^{-1}\pi_5)\br(\pi_1\pi_5^{-1})\br(\pi_2)\br( \pi_3^{-1}\pi_4^r)^{-1}\\
    =&\delta_s\br(\pi_3^{-1}\pi_5)\br(\pi_1)\br(\pi_5^{-1})\br(\pi_2)\br(\pi_3^{-1}\pi_4^r)^{-1}\\
    =&\delta_s\br(\pi_3^{-1}\pi_5)\br(\pi_1)\br(\pi_5^{-1}\pi_2)\br(\pi_3^{-1}\pi_4^r)^{-1}\\
    =&\delta_s\br(\pi_3^{-1}\pi_5)\br(\pi_1\pi_6)\br(\pi_6^{-1})\br(\pi_5^{-1}\pi_2)\br(\pi_3^{-1}\pi_4^r)^{-1}\\
    =&\delta_s\br(\pi_3^{-1}\pi_5\pi_1\pi_6)\br(\pi_6^{-1})\br(\pi_5^{-1}\pi_2\pi_4^{-r}\pi_3)\\
    =&\delta_s\br(\pi_3^{-1}\pi_5\pi_1\pi_6)\br(\pi_6^{-1})\br(\pi_1^{-1}\pi_5^{-1}\pi_3)\\
    =&\delta_s\br(\pi_3^{-1}\pi_5\pi_1\pi_6)\br(\pi_6^{-1}\pi_1^{-1}\pi_5^{-1}\pi_3).
\end{align*}

For each $s-\left\lfloor\frac{s}{r}\right\rfloor+1\le i\le s$, we have $\pi_3(i) =s - r(s-i)\ge r+1$. Thus we have $\pi_5^{-1}(s-r(s-i))=i$ for $s-\left\lfloor\frac{s}{r}\right\rfloor+1\le i\le s$. Because $s-\frac{s}{r}-r+1 = \frac{(s-r)(r-1)}{r}>0$, we have  $s-\left\lfloor\frac{s}{r}\right\rfloor+1\ge r+1$. Therefore, we have $\pi_6^{-1}\pi_1^{-1}\pi_5^{-1}\pi_3(i)=i$ for each $s-\left\lfloor\frac{s}{r}\right\rfloor+1\le i\le s$. Let $\pi_0\in S_{s-\left\lfloor\frac{s}{r}\right\rfloor}$ be the permutation with $\pi_0(i)=\pi_6^{-1}\pi_1^{-1}\pi_5^{-1}\pi_3(i)$ for each $1\le i\le s-\left\lfloor\frac{s}{r}\right\rfloor.$ Then the braid \[\delta_s\br(\pi_3^{-1}\pi_5\pi_1\pi_6)\br(\pi_6^{-1}\pi_1^{-1}\pi_5^{-1}\pi_3)\] is obtained from \[\delta_{s-\left\lfloor\frac{s}{r}\right\rfloor}\br(\pi_0^{-1})\br(\pi_0)\]
via stabilization moves. By Markov's theorem, the closure of latter braid is $T_{r,s}$. By Theorem~\ref{main-1}, the petal number of $T_{r,s}$ is at most $2s- 2\left\lfloor\frac{s}{r}\right\rfloor+1$.

\bibliographystyle{alpha}
\bibliography{ref}
\end{document}